\documentclass{au}
\usepackage{hyperref}
\usepackage{tikz}
\usepackage[english]{babel}
\usepackage[utf8]{inputenc}
\usepackage{amsfonts,amsmath,amssymb,amsthm,enumerate,latexsym}

\let\phi\varphi

\newcommand\tupl[1]{\overline{#1}}

\let\subset\subseteq

\newcommand\AND{\mathrel{\wedge}}

\def\AND{\mathop{\wedge}}

\def\en{{\mathbb N}}

\let\epsilon\varepsilon

\newcommand\algA{{\mathbf A}}
\newcommand\algP{{\mathbf P}}

\theoremstyle{plain}
\newtheorem{theorem}{Theorem}
\newtheorem{lemma}[theorem]{Lemma}
\newtheorem*{lemma*}{Lemma}

\newtheorem{observation}[theorem]{Observation}

\theoremstyle{definition}
\newtheorem{definition}[theorem]{Definition}

\author{Alexandr Kazda}
\email{alex.kazda@gmail.com}
\address{Department of Algebra, Charles University, Sokolovsk\'a 83, 186 75,
Prague, Czechia, ORCID:0000-0002-7338-037X
}
\title[Taylor term does not imply any nontrivial SLEMC]{Taylor term does not imply any nontrivial linear one-equality Maltsev condition}
\thanks{This work was supported by European Research Council under the
European Union's
Seventh Framework Programme (FP7/2007-2013)/ERC grant agreement
no 616160, the PRIMUS/SCI/12 and UNCE/SCI/022 projects
of the Charles University.}
\begin{document}
\maketitle

\begin{abstract}
  It is known that any finite idempotent algebra that satisfies a
  nontrivial Maltsev condition must satisfy the linear
  one-equality Maltsev condition (a variant of the term
  discovered by M.~Siggers and refined by K.~Kearnes, P.~Markovi\'c, and R.~McKenzie):
  \[
    t(r,a,r,e)\approx t(a,r,e,a).
  \]
We show that if we drop the finiteness assumption, the $k$-ary weak near
 unanimity equations imply only trivial linear one-equality
 Maltsev conditions for every $k\geq 3$.
 From this it follows that 
 there is no nontrivial linear one-equality condition that would
  hold in all idempotent algebras having
  Taylor terms.
  
  Miroslav Olšák has recently shown that there is a
  weakest nontrivial strong Maltsev condition for idempotent
  algebras. Olšák has found several such (mutually equivalent)
  conditions consisting of two or more equations.
  Our result shows that Olšák's equation systems can't be compressed
  into just one equation.
\end{abstract}

\section{Introduction}

In this note we show that for every $k\geq 3$ the free algebra
with a $k$-ary weak near unanimity term does not satisfy any
nontrivial linear one-equality Maltsev condition. This is in
contrast to the finite case where having a Taylor term means that
the algebra in question has the Siggers
term~\cite{siggers-original}. The original Siggers term is
equivalent~\cite[Theorem 2.2.]{kearnes-markovic-mckenzie-omit-1}
to the single equation form (the mnemonics for names of variables is due to Ryan O'Donnel):
\[
  t(r,a,r,e)\approx t(a,r,e,a).
\]
Miroslav Olšák has recently shown that having a Taylor term
is a strong Maltsev condition even for infinite (idempotent) algebras~\cite{mirek-omit-1}.
Olšák's shortest condition consists of two linear identities and
it would be natural to ask if one can't do better and use only
one equation. This paper shows that such an improvement is
impossible.

\section{Preliminaries}
An \emph{algebra} $\algA$ consists of a base set $A$ on which
acts a set of \emph{basic operations} of $\algA$. An operation is
a mapping $f\colon A^n\to A$ where $n\in \en$ is the arity of $f$.

The clone of operations of $\algA$ is the smallest set of operations that contains
all basic operations of $\algA$ as well as projections (that is, operations
of the form $\pi_i^k(x_1,\dots,x_k)=x_i$) and is closed under
composing operations. An algebra is
\emph{idempotent} if for any operation $t$ of $\algA$ and any $a\in A$ we have
$t(a,\dots,a)=a$. It is an easy exercise to verify that 
for $\algA$ to be idempotent it suffices that just the basic operations of $\algA$ are
all idempotent.

A \emph{term} of $\algA$ is a (syntactically correct) composition of basic operation symbols of $\algA$
and variables. An  \emph{equational identity}, or \emph{equation}, is a statement of the form ``$u\approx v$'' where $u$ and $v$ are
terms and the symbol ``$\approx$''
stands for ``the left hand side equals the right hand side after any assignment of
members of $A$ to variables.'' An example of an identity is $t(x,\dots,x)\approx x$
which says that the operation $t$ is idempotent.

A \emph{variety} is a class of algebras sharing the same
signature (the same basic operation symbols and arities of basic
operations) that is closed under taking subalgebras, products and
homomorphic images, or equivalently(by Birkhoff's theorem~\cite[Theorem
4.41]{uabook}) a class of algebras defined
by a system of equational identities. The set of all identities
that holds in a variety is called the equational theory of the
variety. 


A \emph{strong Maltsev condition} is a finite list of identities
involving some operation symbols. An algebra
$\algA$ satisfies a strong Maltsev condition $M$ if for each
$k$-ary operation symbol
in $M$ one can choose a $k$-ary operation of $\algA$ so that
when we replace the symbols of $M$ by operations of $\algA$, we
get a system of equations that are all true in $\algA$. A variety satisfies the condition $M$ if all
algebras in the variety satisfy $M$. Since we consider only strong Maltsev
conditions in this paper, we will omit the adjective ``strong'' for brevity.

A Maltsev condition is \emph{trivial} if it is satisfied by the
algebra $\algP$ on two elements $0$ and $1$ whose set of operations consists 
of projections only. An example of a trivial
strong Maltsev condition is $t(t(x,y,z),y,z)\approx t(x,x,z)$; one can satisfy
this condition by choosing $t$ to be the third projection (i.e. $t(a,b,c)=c$
for all $a,b,c$).

A Maltsev condition is called \emph{linear} if its identities don't involve
compositions, i.e. all identities have the form $t(\tupl x)\approx s(\tupl y)$
or $t(\tupl x)\approx z$ or $z\approx r$ where $\tupl x, \tupl y$ are tuples of
variables, $t,s$ are (possibly equal) operation symbols, and $z,r$ are
variables (we include the third case for completeness only; only a trivial
algebra can satisfy an identity of the form $z\approx r$ where $z,r$ are
distinct variables).

Having a \emph{$k$-ary weak near
unanimity operation} ($k$-wnu) for a fixed $k\in \{2,3,4,\dots\}$
is a Maltsev condition that consists of the following $k$ linear identities for the $k$-ary
operation symbol $w$:
\begin{align*}
  w(x,x,x,\dots,x,x)&\approx x\\
  w(y,x,x,\dots,x,x)&\approx  w(x,y,x,\dots,x,x)\approx
  w(x,x,y,\dots,x,x)\approx\cdots\\
  \cdots&\approx w(x,x,x,\dots,y,x) \approx  w(x,x,x,\dots,x,y).
\end{align*}

Having a \emph{Taylor operation (term)} refers to having an operation $t$ satisfying any linear Maltsev
condition of the form
\begin{align*}
  t(x,\dots,x)&\approx x\\
  t(x,?,?,\dots,?)&\approx t(y,?,?,\dots,?)\\
  t(?,x,?,\dots,?)&\approx t(?,y,?,\dots,?)\\
		  &\vdots\\
  t(?,?,?,\dots,x)&\approx t(?,?,?,\dots,y),
\end{align*}
where $x,y$ are variables and the question marks stand for some choice of $x$'s and $y$'s.

It is immediate to see that any operation that is a $k$-wnu is also a Taylor
term (but not the other way around). Finite
idempotent algebras with Taylor terms are well understood as
the following theorem shows.

\begin{theorem}[Combining~\cite{taylor-homotopy}, \cite{maroti-mckenzie-wnu},
  \cite{kearnes-markovic-mckenzie-omit-1},
  and~\cite{barto-kozik-cyclic-terms-and-csp}]\label{thm:finite}
  Let $\algA$ be a finite idempotent algebra. Then the following are
  equivalent:
  \begin{enumerate}
    \item $\algA$ satisfies a nontrivial Maltsev condition,
    \item $\algA$ has a Taylor term,
    \item $\algA$ has a $k$-wnu for some $k\in \en$,
    \item $\algA$ has a $k$-ary cyclic term for some $k\in \en$,
	    where a cyclic term satisfies the equation
	    \[
		    c(x_1,x_2,\dots,x_k)\approx
		    c(x_k,x_1,\dots,x_{k-1}).
	    \]
    \item $\algA$ satisfies the Maltsev condition (known as a Siggers term)
\[
  t(r,a,r,e)\approx t(a,r,e,a)
\]
where $a,e,r$ are variables.
  \end{enumerate}
\end{theorem}

Note that the cyclic and Siggers term conditions, unlike the other equivalent
conditions involve only one identity (plus idempotency, which we assume from the start).
We will abbreviate single linear equality Maltsev condition as
\emph{SLEMC}. Siggers term and cyclic term conditions are
examples of nontrivial SLEMCs, while the 3-wnu condition is not a
SLEMC. 

Our work stems from an attempt to generalize Theorem~\ref{thm:finite} to
infinite idempotent algebras. We will show that there is no
analogue of the last two points,
i.e. that having a $k$-wnu for $k\geq 3$ does not imply a nontrivial SLEMC.  
(For $k=2$, we have the SLEMC $w(x,y)\approx w(y,x)$.)

\section{3-wnu implies only trivial SLEMCs}

In this section we show in detail that having a 3-wnu term does not imply any
nontrivial SLEMC. The general case of having a $k$-wnu differs from the 3-wnu
situation only by a
slightly more complicated notation. This is why we first give the proof for
3-wnu and then, in the next section, we sketch the argument for the general
case without going into details.

\begin{theorem}\label{thm:main}
  Let $V$ be the variety of algebras with one ternary basic
  operation $w$ and with the equational theory generated by the 3-wnu identities
  \begin{align*}
    w(x,x,x)&\approx x\\
    w(x,x,y)&\approx w(x,y,x)\approx w(y,x,x).
  \end{align*}
  This variety (which is idempotent and has a 3-wnu term) 
  does not satisfy any nontrivial SLEMC.
\end{theorem}

The proof of this theorem will occupy the rest of this section.
From the equations, we can see that $V$ is idempotent and that $w$ is a 3-wnu
operation, so the only nonobvious statement is that $V$ does not
satisfy any nontrivial SLEMC.

Since $V$ contains algebras on more than one element (for example
$\{0,1\}$ with $w(x,y,z)=x+y+z \pmod 2$), any candidate for a
nontrivial SLEMC has to have a rather specific shape:

\begin{observation}\label{obs:form}
  Let $\algA$ be an idempotent algebra on at least two elements. If $\algA$
  satisfies a nontrivial SLEMC $M$, then $M$ has the form
\[
  t(x_1,\dots,x_m)\approx t(y_1,\dots,y_m)
\]
where $t$ is an operation symbol and $x_1,\dots,x_m,y_1,\dots,y_m$
are variable symbols such that $x_i\neq y_i$ for $i=1,2,\dots, m$.
\end{observation}
\begin{proof}
Assume that $\algA$ satisfies a nontrivial SLEMC $M$ of the form
\[
  r(x_1,\dots,x_m)\approx s(y_1,\dots,y_k)
\]
where $r,s$ are two different operation symbols. Since the
condition $M$ is
supposed to be
nontrivial, the variable sets $\{x_1,\dots,x_m\}$ and $\{y_1,\dots,y_k\}$ must
be disjoint (for had we $x_i=y_j$ then we could satisfy $M$ by taking $r$ and $s$ to be the projections
to the $i$-th and $j$-th coordinates, respectively). Therefore, the SLEMC $M$
implies
\[
  r(x,\dots,x)\approx s(y,\dots,y)
\]
where $x, y$ are distinct variable symbols. Since $\algA$ is
idempotent, the operations of $\algA$ realizing $r$ and $s$ are idempotent and
$\algA$ satisfies $x\approx y$, meaning $|A|=1$.

A similar argument rules out the SLEMC 
\[
  t(x_1,\dots,x_m)\approx y.
\]
This leaves only the possibility 
\[
  t(x_1,\dots,x_m)\approx t(y_1,\dots,y_m),
\]
where $x_i\neq y_i$ for all $i$ (were $x_i=y_i$, we could satisfy $M$ by taking
$t$ to be the $m$-ary projection to the $i$-th coordinate).
\end{proof}

We will now construct an algera in $V$ that satisfies no nontrivial SLEMC. Two
comments before we begin: First, we are actually going to construct the free
countably generated algebra in $V$. Second, we note for readers familiar with
term rewrite systems (see eg.~\cite{term-rewriting}) that we are implicitly
studying the term rewrite system with the rules $w(x,x,x)\to x$,
$w(y,x,x),w(x,y,x),w(x,x,y)\to u(x,y)$ where $u(x,y)$ is a new symbol that
stands for $w(y,x,x)$. We opted to not use the machinery of term rewriting
because an elementary argument is reasonably short and prepares us for
calculations with normal forms later on.

Let $X$ be a countable set of variable symbols.
Let $T$ be the set of all possible terms we can get using $X$ and
a single ternary operation symbol $w$ (so for example
$w(y,w(x,y,z),y)\in T$).

We define the set $A$ of ``normal forms'' of terms of $T$ modulo the 3-wnu
identities as follows: A term $t$ lies in $A$ if either $t$ is a variable from
$X$, or $t=w(a_1,a_2,a_3)$ where $a_1,a_2,a_3\in A$ and we have $a_1\neq a_2,
a_3$ (for example, $w(w(x,y,z),y,y)$ lies in $A$, but $w(y,w(x,y,z),y)$ does
not).

Let $t\in T$ be a term. It is easy to prove by induction on the number of
occurrences of $w$ in $t$ that we can use the 3-wnu identitites to rewrite $t$
to a term $t'\in A$ such that $t\approx t'$ in $V$ (in fact, the term 
$t'$ is unique for a given $t$; we omit the proof of this as we will not need it). 

Consider the algebra $\algA=(A,w^\algA)$ with the
operation $w^\algA$ defined as
follows:
\begin{enumerate}
  \item If $a_1,a_2,a_3\in A$ are pairwise different then we let
    $w^\algA(a_1,a_2,a_3)=w(a_1,a_2,a_3)$,
  \item if $a,b\in A$ are different then we let all three of
	  $w^\algA(a,a,b)$, $w^\algA(a,b,a)$, $w^\algA(b,a,a)$ to
	  be equal to $w(b,a,a)$, and
  \item if $a\in A$, then $w^\algA(a,a,a)=a$.
\end{enumerate}

It is easy to verify that $\algA$ is closed under $w^\algA$. Observe also that
the operation $w^\algA$ is a 3-wnu operation, so $\algA\in V$.

Note that in many cases we have $w^\algA(a,b,c)=w(a,b,c)$, but this is not
always true. This is why we distinguish between $w^\algA$ (operation
symbol of $\algA$) and $w$ (formal symbol used to
describe terms of $V$).

Since $\algA\in V$, to prove Theorem~\ref{thm:main} it is
enough to show that $\algA$ satisfies only trivial SLEMCs.
To that end let $R$ be the subalgebra of $\algA^2$ generated by $\{(x,y)\colon x,y\in
X,\, x\neq y\}$. The following observation shows that to prove
Theorem~\ref{thm:main}, it is enough to show that $R$ does not intersect the
diagonal.

\begin{observation}\label{obs:SLEMC}
  If the variety $V$ satisfies a nontrivial SLEMC, then the
  relation $R$ defined above
  intersects the diagonal of $A^2$ (in other words, there is an $r\in A$ such that
  $(r,r)\in R$).
\end{observation}
\begin{proof}
  Assume that $V$ satisfies the SLEMC
  \[
    t(y_1,\dots,y_m)\approx t(z_1,\dots,z_m)
  \]
  where $y_1,\dots,y_m,z_1,\dots,z_m$ are variables (without loss
  of generality) taken from the set $X$. Let us denote by $t^\algA$ the term of
  $\algA$ we
  obtain from $t$ by replacing each symbol $w$ by $w^\algA$. In $\algA$ we thus
  have the equality (note that $y_i$'s and $z_i$'s are members of $A$ here, not
  variable symbols).
  \[
    t^\algA(y_1,\dots,y_m)=t^\algA(z_1,\dots,z_m)=r.
  \]
for some $r\in A$.

  We have $(y_i,z_i)\in R$ for all $i$ and so applying the operation $t^\algA$
  to pairs $(y_1,z_1),\dots,(y_m,z_m)\in R$ gets us $(r,r)\in R$.
\end{proof}

While we would like to show that $R$ does not intersect the
diagonal, idempotency prevents us from comfortably doing a proof by
induction on term complexity on $R$ itself. This is why we take a detour
through subterms.

\begin{definition}
  We define the relation ``to be a subterm'' on the set $A$, denoted by $\preceq$, as the
  reflexive and transitive closure of the set
  \[
    Q=\{(a,b)\colon a,b\in A,\,\exists c,d,e\in A,\,
    b=w(c,d,e),\,a\in\{c,d,e\}\}.
  \]
\end{definition}
Informally, $a\preceq b$ if in the term $b$ we can find a subterm that is identical to $a$. 
Note that $\preceq$ is defined using the
(syntactic) symbol $w$. However, it turns out that $\preceq$ behaves well with respect to
the operation $w^\algA$, too:

\begin{lemma}\label{lem:subterms}
	The following holds for $\preceq$:
  \begin{enumerate}[(a)]
    \item\label{itm:xy} If $x,y$ are distinct members of $X$
	    (i.e. variables), then $x\not\preceq y$.
    \item For all $b,c,d\in A$, we have $b\preceq w^\algA (b,c,d)$,
      $w^\algA(d,b,c)$, $w^\algA(c,d,b)$.
    \item\label{itm:wA} For all $a,b,c,d\in A$ such that
	    $a\preceq b$ we have $a\preceq w^\algA
      (b,c,d)$, $w^\algA(d,b,c)$, $w^\algA(c,d,b)$.
    \item\label{itm:eq}If $a,b,c,d\in A$ are such that $a\preceq
	    w^\algA(b,c,d)$ and $a\not\preceq b,c,d$, then $a=w^\algA(b,c,d)$.
  \end{enumerate}
\end{lemma}
\begin{proof}
  \begin{enumerate}[(a)]
    \item Since $x\neq y$, the only way we could have had $x\preceq y$ would be if there
      was a chain of $k\geq 2$ terms $x=t_1,t_2,\dots,t_k=y$ such that
      $(t_i, t_{i+1})\in Q$ for $i=1,\dots,k-1$ (where $Q$ is the set
      from the definition of $\preceq$). From this we get
      that for all $i=1,\dots,k-1$ we have
      $t_{i+1}=w(p_i,q_i,r_i)$ where $p_i,q_i,r_i$ are 
      members of $A$ and $t_i$ appears at least once in
      $(p_i,q_i,r_i)$. By induction on $i$, it
      follows that each $t_i$ must have at least $i-1$ occurrences of the symbol
      $w$, so $t_k$ contains at least one symbol $w$. But the term $t_k=y$ has no
      $w$ in it, a contradiction.

    \item We will show $b\preceq w^\algA(b,c,d)$; the other two subterm
      relationships are similar. 
      
      Unless $b=c=d$, we have $w^\algA(b,c,d)\in
      \{w(b,c,d),w(d,b,c),w(c,d,b)\}$ and $b$ is a subterm of each of
      the terms on the right side, so we are done. In the case $b=c=d$, we get
      $w^\algA(b,c,d)=b$ and $b\preceq b$ follows from the reflexivity of
      $\preceq$.
    \item This follows from the transitivity of $\preceq$ and the previous
      point: We have $a\preceq b\preceq w^\algA
      (b,c,d),w^\algA(d,b,c),w^\algA(c,d,b)$.
    \item Were $b,c,d$ all equal, we would have $a\preceq w^\algA(b,b,b)=b$, a
      contradiction. with $a\not\preceq b$. Therefore, without loss of generality
      $w^\algA(b,c,d)=w(b,c,d)$ (we can reorder $b,c,d$).

      Assume for a contradiction that $a\neq w^\algA(b,c,d)=w(b,c,d)$. Since the
      relation $a\preceq w(b,c,d)$ is not a consequence of reflexivity,
      there is a $k\geq 2$ and a chain $a=t_1\preceq t_2\preceq
      \dots\preceq t_{k-1}\preceq t_k=w(b,c,d)$ witnessing $a\preceq
      w(b,c,d)$ with $(t_i,t_{i+1})\in Q$ for all $i$. But then $a\preceq t_{k-1}$ and $t_{k-1}$ needs to be
      one of $b,c,d$ by the definition of $Q$, a
      contradiction with
      $a\not\preceq b,c,d$.
  \end{enumerate}
\end{proof}

Let now $S$ be the following relation on $A$:
\[
  S=\{(a,b)\in A^2\colon a\not \preceq b \AND b\not\preceq a\}.
\]
By part~(\ref{itm:xy}) of Lemma~\ref{lem:subterms}, the generators of $R$ lie in
$S$ and since $\preceq$ is reflexive, $S$ does not intersect the
diagonal.

\begin{lemma}\label{lem:cases}
  The relation $S$ is a subuniverse of $\algA^2$.
\end{lemma}
\begin{proof}
  Let us take $(a_1,b_1),(a_2,b_2),(a_3,b_3)\in S$ such
  that (without loss of generality) 
  $w^\algA(a_1,a_2,a_3)\preceq w^\algA(b_1,b_2,b_3)$.
  
  We consider several cases:
  \begin{enumerate}[(a)]
    \item Assume that $w^\algA(a_1,a_2,a_3)\neq w^\algA(b_1,b_2,b_3)$. Then by
      part~(\ref{itm:eq}) of Lemma~\ref{lem:subterms},
      $w^\algA(a_1,a_2,a_3)$ (which plays the role of $a$ in the
      Lemma) needs to be a subterm of one of $b_1,b_2,b_3$.
      Without loss of generality assume $w^\algA(a_1,a_2,a_3)\preceq b_1$.
      But $a_1\preceq w^\algA(a_1,a_2,a_3)$ by
      part~(\ref{itm:wA}) of
      Lemma~\ref{lem:subterms}. We have $a_1\preceq
      w^\algA(a_1,a_2,a_2)\preceq b_1$, which is a contradiction with $(a_1,b_1)\in S$.
    \item Assume that $w^\algA(a_1,a_2,a_3)= w^\algA(b_1,b_2,b_3)$ and
      $a_1=a_2=a_3$. Then $b_1\preceq
      w^\algA(b_1,b_2,b_3)=w^\algA(a_1,a_1,a_1)=a_1$ (where the subterm
      relationship follows again by part~(\ref{itm:wA}) of
      Lemma~\ref{lem:subterms}) and therefore $b_1\preceq a_1$.

      The same argument takes care of the case $b_1=b_2=b_3$.
    \item Assume that $w^\algA(a_1,a_2,a_3)= w^\algA(b_1,b_2,b_3)$ and
      $\{a_1,a_2,a_3\}=\{c,d\}$ with $d$ appearing twice, i.e. $w^\algA(a_1,a_2,a_3)= w(c,d,d)= w^\algA(b_1,b_2,b_3)$
      
      Since $b_1,b_2,b_3$ are not all equal, by the definition of $w^\algA$
      we must have $\{b_1,b_2,b_3\}=\{c,d\}$ with $d$ appearing twice. Since we
      have three pairs $(a_1,b_1)$, $(a_2,b_2)$, $(a_3,b_3)$ and only two
      appearances of $c$ (one for $a$'s, one for $b$'s), it follows that
      there exists an $i$ such that $a_i=d=b_i$. However, $(d,d)\not \in S$, 
      a contradiction.
    \item Assume that $w^\algA(a_1,a_2,a_3)= w^\algA(b_1,b_2,b_3)$ and
      $a_1,a_2,a_3$ are pairwise different, i,e. $w(a_1,a_2,a_3)=
      w^\algA(b_1,b_2,b_3)$. Since $b_1,b_2,b_3$ are not all equal, the only
      way to get equality here is to have $a_i=b_i$ for all $i=1,2,3$, a
      contradiction with $(a_i,b_i)\in S$.
  \end{enumerate}
\end{proof}

\begin{proof}[Proof of Theorem~\ref{thm:main}]
  By Lemma~\ref{lem:cases}, we get that $S$ is a subuniverse of
$\algA^2$ that contains all generators of $R$ and thus $R\subset S$. As we have
seen above, $S$ is disjoint from the diagonal, so $R$ must be disjoint from the
diagonal. Therefore, by Observation~\ref{obs:SLEMC}, the variety $V$ can't
satisfy a nontrivial SLEMC.
\end{proof}

\section{$k$-wnu implies only trivial SLEMCs}
\begin{theorem}
  For any $k\geq 3$ the $k$-wnu identities don't imply a nontrivial SLEMC.
\end{theorem}
\begin{proof}
  The proof is very similar to the proof of
  Theorem~\ref{thm:main}, so we only sketch the main points here. 
  
  We take $V$ to be the variety defined by the $k$-wnu
equations for a $k$-wnu operation $w$, $X$ a countable set of variables and $A$ the smallest
  set of terms made from $X$ and $w$ such that $X\subset A$ and
  $w(a_1,\dots,a_k)\in A$ if and only if $a_1,\dots,a_k\in A$ and there are at
  least two
	  distinct indices $i,j$ such that $a_1\neq a_i,a_j$.

  Again, we consider the algebra $\algA=(A,w^\algA)\in V$ with $w^\algA(a_1,\dots,a_k)$ defined as
\[
  w^\algA(a_1,\dots,a_k)=\begin{cases}
    a_1& a_1=a_2=\dots=a_k\\
    w(c,d,d,\dots,d)& \exists i\in\{1,\dots,k\}, a_i=c,\\ 
    &a_1=a_2=\dots=a_{i-1}=a_{i+1}=\cdots\\
    &\cdots=a_k=d\neq c\\
    w(a_1, a_2,\dots,a_k)&\text{otherwise.}
  \end{cases}
\]
The relation $R$ is again generated in $\algA^2$ by
$\{(x,y)\in X\colon x\neq y\}$, while the subterm relation $\preceq$
is defined as the reflexive and transitive closure of
  \[
    \{(a,b)\colon a,b\in A,\,\exists c_1,c_2,\dots,c_k\in A,\,
    b=w(c_1,\dots,c_k),\,a\in\{c_1,\dots,c_k\}\}.
  \]
As before, we show that 
\[
  S=\{(a,b)\in A^2\colon a\not \preceq b \AND b\not\preceq a\}
\]
is $\algA$-invariant and thus prove that $R$ does not intersect
the diagonal which implies that $V$ satisfies no nontrivial SLEMC.
\end{proof}

\bibliographystyle{plain}

\bibliography{citations}
\end{document}